\documentclass[reqno,11pt]{amsart}

\usepackage{amscd}
\usepackage{amsmath}
\usepackage{amssymb}
\usepackage[american]{babel}
\usepackage{bbm}
\usepackage{bookmark}
\usepackage{cmap} 
\usepackage{dsfont}
\usepackage{enumerate}
\usepackage{epigraph}
\usepackage[mathscr]{euscript}
\usepackage[myheadings]{fullpage}
\usepackage{graphicx}
\usepackage[geometry]{ifsym}
\usepackage{mathabx}
\usepackage{mathrsfs}
\usepackage{mathtools}
\usepackage{refcount}
\usepackage{setspace}
\usepackage{stmaryrd}
\usepackage{thinsp}
\usepackage[safe]{tipa}
\usepackage{verbatim}
\usepackage[all,2cell]{xy} \UseTwocells
\usepackage{xr}
\usepackage[multiple]{footmisc}
\usepackage{hyperref}

\numberwithin{equation}{subsection}

\newtheorem{thm}{Theorem}
\newtheorem*{thm*}{Theorem}

\newtheorem*{cor*}{Corollary}

\newtheorem{prop-const}[thm]{Proposition-Construction}

\newtheorem*{conjecture*}{Conjecture}

\newtheorem*{princ*}{Principle}

\theoremstyle{remark}

\newcounter{steps}[thm]
\newtheorem{step}[steps]{Step}

\newcommand{\xar}[1]{\xrightarrow{#1}}

\newcommand{\bA}{{\mathbb A}}
\newcommand{\bB}{{\mathbb B}}

\newcommand{\bG}{{\mathbb G}}

\newcommand{\bZ}{{\mathbb Z}}

\newcommand{\fm}{{\mathfrak m}}

\newcommand{\on}{\operatorname}

\newcommand{\mathendash}{\text{\textendash}}

\newcommand{\Ext}{\on{Ext}}

\newcommand{\Spec}{\on{Spec}}

\newcommand{\Tor}{\on{Tor}}

\newcommand{\Vect}{\mathsf{Vect}}

\newcommand{\Sym}{\on{Sym}}

\renewcommand{\mod}{\mathendash\mathsf{mod}}

\newcommand{\colim}{\on{colim}}

\newcommand{\ShvCat}{\mathsf{ShvCat}}

\renewcommand{\lim}{\on{lim}}

\newcommand{\heart}{\heartsuit}

\newcommand{\QCoh}{\mathsf{QCoh}}

\renewcommand{\subset}{\subseteq}

\makeatletter
\newcommand{\biggg}{\bBigg@{4}}
\newcommand{\Biggg}{\bBigg@{5}}
\makeatother

\date{\today}

\begin{document}

\frenchspacing

\setlength{\epigraphwidth}{0.9\textwidth}
\renewcommand{\epigraphsize}{\footnotesize}

\title{Acyclic complexes and 1-affineness}

\author{Dennis Gaitsgory and Sam Raskin}

\address{Harvard University, 
1 Oxford St, Cambridge, MA 02138.} 
\email{gaitsgde@math.harvard.edu}

\address{Massachusetts Institute of Technology, 
77 Massachusetts Avenue, Cambridge, MA 02139.} 
\email{sraskin@mit.edu}

\begin{abstract}

This short note is an erratum to \cite{shvcat}, correcting the
proof of one of its main results. 
It includes some counterexamples regarding infinite-dimensional
unipotent groups and affine spaces that may be of independent interest. 

\end{abstract}

\maketitle


\setcounter{section}{1}

\subsection{} 

In \cite{shvcat} Theorem 2.4.5, one finds the
claim that $\bA^{\infty} \coloneqq \colim_n \bA^n$
is not\footnote{Recall, however, that there are some remarkably similar
spaces that \emph{are} 1-affine. For example, \cite{shvcat}
Theorem 2.6.3 says that $\bA_{dR}^{\infty}$ is 1-affine. This
result is generalized (in a more sophisticated setup) in \cite{locsys}.}
1-affine,\footnote{We refer to \emph{loc. cit}. for 
the definition of this notion. We also follow the notational conventions
of \emph{loc. cit}., so everything should be understood in the
derived sense.} 
where we work relative to a field
$k$ of characteristic $0$. The proof in \emph{loc. cit}.
is not correct, and the purpose of this note is to correct it.

Along the way, we also give some (possibly new) 
counterexamples on representations of
$\bA^{\infty}$ considered as additive group,
and the formal completion of a pro-infinite dimensional
affine space at the origin. These counterexamples will be a categorical
level down, i.e., they will reveal pathological behavior
for $\QCoh$ rather than $\ShvCat$.

\subsection{What is the problem in \cite{shvcat}?}

In \S 9.4.3 of \emph{loc. cit}., there is a claim that
something is ``easy to see," with no further explanation.
In fact, the relevant claim is not true. This failure invalidates
the argument about 1-affineness given in \emph{loc. cit}.

\subsection{Structure of the argument}

There has only ever been one successful strategy to proving
that a prestack is not 1-affine: showing that 1-affineness would imply 
some functor is co/monadic (usually by computing some tensor
product of DG categories and applying the Beck-Chevalley formalism),
and then showing that the relevant functor is not conservative.

We will exactly follow this strategy. Namely, in 
\S \ref{ss:cons-start}-\ref{ss:cons-finish}, we prove two 
results, Theorems \ref{t:ainfty-reps} and \ref{t:pro-ainfty},
about the non-conservativeness of various functors. 
Then in \S \ref{ss:main}, we deduce that $\bA^{\infty}$ is not 1-affine.

The reader who is most invested in the non-1-affineness of $\bA^{\infty}$
might prefer to read the statement
of Theorem \ref{t:ainfty-reps} and then skip ahead to \ref{ss:main},
returning to read the proofs of the other results after seeing
their application.

\subsection{Invariants and ind-unipotent groups}\label{ss:cons-start}

Our first main result is the following. 

\begin{thm}\label{t:ainfty-reps}

$\Gamma: \QCoh(\bB \bA^{\infty}) \to \Vect$ is not conservative. 

\end{thm}

In other words, the functor of 
invariants for ind-unipotent groups is not conservative,
even though it is for unipotent groups. 

Moreover, the construction will 
produce a non-zero representation in the heart of the\footnote{Note
that $\bB \bA^{\infty}$ is ind-smooth, in particular
ind-flat, and therefore its $\QCoh$ indeed has an obvious $t$-structure.} 
$t$-structure whose
invariants vanish, and whose higher group cohomologies vanish 
as well. That is to say, this is an essential issue, not the sort resolved
by some kind of renormalization procedure. 

\subsection{}

We will deduce Theorem \ref{t:ainfty-reps} from the next result.

Let $A = k[t_1,t_2,\ldots]$, so $\Spec(A)$ is a pro-infinite
dimensional affine space.\footnote{We avoid geometric notation
here so that $\bA^{\infty}$ has a unique meaning in this text.}

\begin{thm}\label{t:pro-ainfty}

There exists $0 \neq V \in A\mod^{\heart}$ with the properties that:

\begin{itemize}

\item Each $t_i$ acts on $V$ nilpotently (in fact, with square zero).

\item $\Tor_j^{A}(V,k) = 0$ for all $j$, where $k$ is equipped with
the $A$-module structure where each $t_i$ acts by zero. 

\end{itemize}

\end{thm}

The proof can be found in \S \ref{ss:cons-finish}.

\begin{proof}[Proof that Theorem \ref{t:pro-ainfty} implies 
Theorem \ref{t:ainfty-reps}]

Since each $t_i$ acts on $V$ nilpotently,\footnote{
And not merely \emph{locally} nilpotently.}
the induced action on $V^{\vee} \neq 0$ is also nilpotent.
Moreover, all of the operators $t_i$ acting on $V^{\vee}$ commute.
Therefore, $V^{\vee}$ has a canonical structure
as an object of $\QCoh(\bB \bA^{\infty})^{\heart}$,
since this is the abelian category of a non-derived vector
spaces equipped with $\bZ^{> 0}$-many
locally nilpotent and pairwise commuting operators.

Finally, observe that:

\[
H^j(\Gamma(\bB \bA^{\infty},V^{\vee})) = 
\Ext^j_{\QCoh(\bB \bA^{\infty})}(k,V^{\vee}) =
\Tor_j^A(k,V)^{\vee} = 0
\]

\noindent since the complex computing these $\Ext$s is dual
to the complex computing the $\Tor$s. This gives the claim.

\end{proof}

\subsection{Construction of $V$, and a heuristic}

We now construct $V$ from Theorem \ref{t:pro-ainfty}, 
and explain why the relevant $\Tor_0$ vanishes.

Namely, $V$ has a basis $v_S$ indexed by finite subsets 
$S \subset \bZ^{>0}$, and we define 
$t_i \cdot v_S = v_{S \setminus \{i\}}$ if $i \in S$
and $t_i \cdot v_S = 0$ if $i \not \in S$. Clearly $t_i^2$ acts
by zero on $V$.

To see that $\Tor_0^A(V,k) = 0$, it suffices to see that
there are no morphisms $f:V \to k$ in $A\mod^{\heart}$.
For any such $f$, we should show that $f(v_S) = 0$ for all $S$ as above. 
Since $S$ is finite
and $\bZ^{>0}$ is infinite, we can find $i \not \in S$.
Then $t_i v_{S \cup \{i\}} = v_S$, so 
$f(v_S) = t_i f(v_{S \cup i}) = 0$ because each $t_i$ acts on $k$ by zero.
(A similar argument also works for $\Tor_1$.)

\subsection{}\label{ss:cons-finish}

We now show that the higher $\Tor$s vanish as well.

\begin{proof}[Proof of Theorem \ref{t:pro-ainfty}]

\step 

First, we rewrite the representation $V$ in a more conceptual
way.

For each $n>0$, let $I_n \subset A$ be the ideal generated
by $t_1^2,\ldots,t_n^2$ and all $t_i$ for $i >n$.

Then we claim $V = \colim_n A/I_n$, where the structure morphisms:

\[
A/I_n \to A/I_{n+1}
\]

\noindent are given by multiplication
by $t_{n+1}$. Indeed, the relevant
structure morphism sends 
$A/I_n$ to $V$ by
mapping the ``vacuum vector" $1 \in A/I_n$ to
$v_{\{1,\ldots,n\}}$ in the right hand side,
which is an injection whose image identifies
with the subspace of $V$ spanned by $v_S$
for $S \subset \{1,\ldots n\} \subset \bZ^{>0}$.

\step 

Next, let $\fm = (t_1,t_2,\ldots) \subset A$, and
observe that the morphism:\footnote{Here and always we use
$\underset{A}{\otimes}$ for the derived tensor product,
i.e., for $\overset{L}{\underset{A}{\otimes}}$.}

\[
A/I_n \underset{A}{\otimes} A/\fm \xar{t_{n+1}\cdot }
A/I_{n+1} \underset{A}{\otimes} A/\fm
\]

\noindent is zero, i.e., is canonically nullhomotopic;
indeed, this follows from the fact that multiplication by 
$t_{n+1}$ is zero on $A/\fm$.

Passing to the colimit over $n$ and using our earlier expression
for $V$, we see that $V \otimes_A A/\fm = 0$ as desired.

\end{proof}

\subsection{Application to non-1-affineness of $\bA^{\infty}$}\label{ss:main}

We now show that $\bA^{\infty}$ is not 1-affine.

\begin{proof}[Proof of Theorem 2.4.5 from \cite{shvcat}]

We follow the beginning of the proof of the theorem from
\cite{shvcat} 9.3.2.

Namely, it suffices to show that the canonical functor:

\[
\Vect \underset{\QCoh(\bA^{\infty}) }{\otimes} \Vect \to 
\QCoh(\Spec(k) \underset{\bA^{\infty}}{\times} \Spec(k))
\]

\noindent is not an equivalence. 

Let $\bG_m$ act on $\bA^{\infty}$ by scaling. The above
functor is a morphism of $\QCoh(\bG_m)$-module\footnote{Here
$\QCoh(\bG_m)$ is equipped with the convolution monoidal
structure.}
categories. 

By 1-affineness of $\bB \bG_m$, it is equivalent to show that the
above functor is not an equivalence
after passing to $\bG_m$-equivariant categories.
Then using the ``shift of grading trick" (c.f. \cite{arinkin-gaitsgory} A.2)
and 1-affiness of $\bB \bG_m$ again, we see that it is equivalent
to show that the functor:

\[
\Vect \underset{\underset{n}{\lim} \big(\Sym(k^{\oplus n}[-2])\mod\big)}{\otimes} \Vect \to
\underset{n}{\lim} (k \underset{\Sym(k^{\oplus n}[-2])} {\otimes} k)\mod =
\underset{n}{\lim} \Sym(k^{\oplus n}[-1])\mod =
\QCoh(\bB \bA^{\infty})
\]

\noindent is not an equivalence, where all structure maps in the
limits are induced
by the projections $k^{\oplus n+1} \to k^{\oplus n}$ (so are given
by tensor product functors everywhere).

By the Beck-Chevalley formalism, the (discontinuous)
right adjoint to the canonical functor:

\[
\Vect = \Vect \underset{\Vect}{\otimes} \Vect \to 
\Vect \underset{\underset{n}{\lim} \Sym(k^{\oplus n}[-2])\mod}{\otimes} \Vect
\]

\noindent (which is induced by the canonical functors
$\Vect \to \Sym(k^{\oplus n}[-2])\mod$ sending $k$ to 
to the free module $\Sym(k^{\oplus n}[-2])$, i.e., the trivial
representation in $\QCoh(\bB \bA^n)$) is monadic
(c.f. \cite{shvcat} Lemma 9.3.3). 

Therefore, it suffices to show that the right adjoint to the
corresponding functor:

\[
\Vect \to 
\underset{n}{\lim} (k \underset{\Sym(k^{\oplus n}[-2])} {\otimes} k) \mod =
\QCoh(\bB \bA^{\infty})
\]

\noindent is not conservative (in particular, not monadic).
But this functor corresponds to group cohomology, and Theorem
\ref{t:ainfty-reps} says that it is not conservative.

\end{proof}

\bibliography{bibtex}{}
\bibliographystyle{alphanum}

\end{document}